\documentclass[11pt,a4paper]{article}
\usepackage{amsmath}
\usepackage{amssymb}
\usepackage{amsthm}
\usepackage{amsfonts}
\usepackage{enumerate}
\usepackage{verbatim}
\usepackage{hyperref}
\usepackage{breakurl,url}
\usepackage{bm}
\usepackage{color}
\usepackage{mathtools} 
\usepackage{pstricks}
\usepackage{pst-node}
\usepackage{pst-plot}
\usepackage{algorithm}
\usepackage{algorithmic}
\usepackage{pst-poly}
\usepackage{caption,subcaption} 
\usepackage[margin=1.3in]{geometry}
\usepackage{mathtools}
\usepackage{caption,subcaption} 
\usepackage{multirow}           
\usepackage[margin=1.3in]{geometry}
\newtheorem{theorem}{Theorem}

\newtheorem{lemma}{Lemma}

\theoremstyle{definition}

\newtheorem{example}{Example}
\newtheorem{remark}{Remark}

\begin{document}
	\title{Iterative Methods for Solving New Generalized Absolute Value Equations}

	\author{Shubham Kumar$^{a,1}$ and Joydip Dhar$^{a,2}$\\
		\emph{\small $^{a}$Department of Engineering Sciences,}
		\emph{\small Atal Bihari Vajpayee - Indian Institute of Information Technology}\\
		\emph{\small and Management Gwalior, Madhya Pradesh, 474015, India}\\
		\emph{\small $^1$Email: shub.srma@gmail.com, $^2$Email: jdhar@iiitm.ac.in}}
	\date{}
	\maketitle

	
	\begin{abstract} 
		
		
		In this paper, we propose two iterative methods, the modified Picard method (MPM) and the generalized Traub method (GTM), to solve a new generalized absolute value equation (NGAVE). We establish their convergence under appropriate assumptions. Additionally, numerical examples demonstrate the effectiveness and applicability of the developed algorithms.
		
		
		
	\end{abstract}
	

	
	\section{Introduction}
	In this study, we investigate numerical methods for solving the following NGAVE
	\begin{equation} \label{Equ1}
		Ax -\vert Bx \vert =d,
	\end{equation}
	where $A,B \in \mathbb{R}^{n\times n},$ $d \in \mathbb{R}^{n}$ are given and $x \in \mathbb{R}^{n}$ is the unknown vector to be determined. 
	The notation $|\cdot|$ denotes the componentwise absolute value operator; that is, $(|Bx|)$ is the vector whose entries are the absolute values of the corresponding components of Bx.
	
	\noindent A special case of the NGAVE (\ref{Equ1}) is the absolute value equation (AVE), obtained by taking $B=I$ (identity matrix):
	\begin{equation} \label{Equ2}
		Ax -\vert x \vert =d.
	\end{equation}
	
	
	The concept of AVE was introduced by Rohn \cite{JiriRohn2004}, who established an alternative theorem concerning their unique solvability. Subsequent studies further developed the theory of unique solvability of AVE in \cite{Mangasarian2006,Mangasarian2007}.The AVE arise in various applications, including engineering mathematics, optimization, linear complementarity problems, and game theory. 
	For comprehensive discussions of various types of AVE, their solvability properties, and applications, we refer the reader to the recent survey articles \cite{Hladik2025Survey,ShubhamOpt.Lett} and the PhD thesis \cite{KumarPhDThesis2024}.
	
	The NGAVE (\ref{Equ1}) generalizes the well-known AVE (\ref{Equ2}). Wu \cite{ShiLiangWu2021NGAVE} first introduced the NGAVE and investigated necessary and sufficient conditions for its unique solvability. Furthermore, the author reformulated the NGAVE as an equivalent vertical linear complementarity problem (VLCP). By exploiting various matrix properties, including the \(P\)-matrix property, singular values, spectral radius, and the row-\(W\) property, Wu \cite{ShiLiangWu2021NGAVE} established necessary and sufficient conditions for the unique solvability of the NGAVE.
	Motivated by Wu's work \cite{ShiLiangWu2021NGAVE}, Kumar et al. \cite{Kumar2022YJOR} introduced a new class of AVE of the form \(Ax-|Bx-c|=d\) and established conditions for their unique solvability.
	For the numerical solution of the NGAVE, Yang et al. \cite{ShiLiangWuNGAVE_SOR2023} proposed an SOR-like iterative method. Subsequently, Khaldi et al. \cite{KhaldiNGAVE2024} developed a Picard iteration method and a generalized Newton method for solving the NGAVE, establishing the global linear convergence of both methods.
	
	
	To the best of our knowledge, only a limited number of numerical methods have been developed for solving the NGAVE.
	Therefore, this paper aims to introduce two iterative methods for its solution. Under suitable conditions, we establish that both methods are well-defined and converge to the unique solution of the NGAVE. Additionally, numerical experiments are conducted to evaluate the performance of the proposed methods and to compare their results with those of existing numerical schemes.
	We also solved the hydrodynamic equilibrium problem with our proposed methods.

	
	
	The paper is structured as follows. Section 2 introduces the basic notation and preliminary results needed for the subsequent analysis. In Section 3, we introduce two iterative schemes and establish their convergence properties. 
	Section 4 presents the numerical experiments and evaluates the proposed methods by comparing their performance with existing approaches. Finally, Section 5 concludes the discussion.
	
	\noindent The existing literature on the numerical methods and unique solvability for the NGAVE is summarized in Table~\ref{tab:comparison50}.

	\begin{table}[htbp]
		\centering
		\caption{Summary of the Numerical Methods and Unique Solvability Conditions for the NGAVE}
		\label{tab:comparison50}
		\small
		\begin{tabular}{p{3.5cm}c p{4cm} p{5cm}}
			\hline
			Author(s) & Year & Work & Results \\
			\hline
			Wu \cite{ShiLiangWu2021NGAVE}
			& 2021
			& Necessary and Sufficient Unique Solvability Conditions
			& Based on P matrix, singular values, row-W property and spectral radius  \\
			
			Yang and Wu \cite{ShiLiangWuNGAVE_SOR2023}
			& 2023
			& Iterative Method
			& SOR-like Method \\
			
			Khaldi and Achache \cite{KhaldiNGAVE2024}
			& 2024
			&  Sufficient Conditions and Iterative Methods
			&  Based on Positive definite matrix and singular values, Picard's Method and Generalized Newton Method \\
			
			\textbf{Proposed Work}
			& 2026
			& Iterative Methods
			& Modified Traub's Method and Modified Picard's Method \\
			\hline
		\end{tabular}
	\end{table}
	
	
	
	
	\section{Preliminaries} \label{section_Preliminaries}
	This section introduces the notation used throughout the paper and presents the preliminary definitions and results needed for the later discussion.
	We utilize the following symbols:
	\begin{itemize}
		\item $I$ denotes the identity matrix.
		
		\item $\sigma_{\max}$ and $\sigma_{\min}$ denote the maximum and minimum singular values, respectively.
		
		\item $\rho(A)$ denotes the spectral radius of the matrix $A$.
		
		\item $\|\cdot\|$ denotes the Euclidean (or 2-) norm.
		
		
		\item For a vector \(x\), \(\operatorname{diag}(x)\) denotes the diagonal matrix with the components of \(x\) as its diagonal entries.
		
		
		\item $\operatorname{sign}(x)$ denotes the sign vector of $x$, whose $k$th component is $-1$, $0$, or $1$ if $x_k<0$, $x_k=0$, or $x_k>0$, respectively.
		
		\item $D$ is diagonal matrix whose diagonal entries comes from the set \{0, 1,-1\}. Further, $D(x)x=|x|$, $D(Bx_{k})=diag(sign(Bx_k))$ and $D(Bx_{k})Bx_{k}=|Bx_{k}|.$
	\end{itemize}

	\begin{theorem} \cite{KhaldiNGAVE2024,ShiLiangWu2021NGAVE} \label{Result 1}
		If any one of the following conditions holds, then the NGAVE (\ref{Equ1}) admits a unique solution for every $d$.\\
		(i) $\sigma_{max}(B)$ $<$ $\sigma_{min}(A)$;  \\
		(ii) $\sigma_{min}(AB^{-1}) > 1$ or $\sigma_{max}(BA^{-1})<1$;\\
		(iii) $||A^{-1}||||B||<1$;\\
		(iv) $\rho(|BA^{-1}|)<1.$
	\end{theorem}
	
	\begin{theorem} \cite{Mangasarian2009} \label{Result 2}
		The inequality $\vert \vert |x|-|y| || \leq 2||x-y||$ holds for each $x,y \in \mathbb{R}^{n}$.
	\end{theorem}	
	
	\begin{theorem} \cite{Ortega1970} \label{Result 3}
		\textbf{Neumann Lemma.} If $I-B$ is invertible whenever $||B||<1,$ then $||(I-B)^{-1}|| \leq \frac{1}{1-||B||}.$
	\end{theorem}
	\section{Proposed Methods}
	In this section, we present the proposed methods, which are divided into two parts. 
	In the first part, we will cover the generalized Traub's method (GTM) and its convergence analysis. In the second part, we will examine the modified Picard's method (MPM)  and its convergence analysis.
	

	\noindent Traub \cite{Traub1964} introduced the following two-step iterative scheme for solving the nonlinear operator equation \(F(x)=0\):
	\[
	\begin{cases}
		y_k &= x_k-[F'(x_k)]^{-1}F(x_k),\\
		x_{k+1} &= y_k-[F'(x_k)]^{-1}F(y_k).
	\end{cases}
	\qquad k=0,1,2,\ldots
	\]
	Haghani \cite{Traub2015} employed this scheme to obtain numerical solutions of the AVE (\ref{Equ2}). Motivated by this approach, we extend the scheme to the NGAVE (\ref{Equ1}) and develop a corresponding iterative method for computing its solution.
	
	
	\subsection{Generalized Traub's Method} \label{section_Traub}
	
	
	Consider the following equation
	\begin{equation}
		F(x)=Ax-|Bx|-d.
	\end{equation}
	The classical Traub's method \cite{Traub1964} for solving the nonlinear equation
	\(
	F(x)=0
	\)
	is given by
	\begin{equation}
		\begin{aligned}
			y_k &= x_k-[F'(x_k)]^{-1}F(x_k),\\
			x_{k+1} &= y_k-[F'(x_k)]^{-1}F(y_k),
		\end{aligned}
	\end{equation}
	where \(F'(x_k)\) denotes the generalized Jacobian of \(F\)  at \(x_k\).
	
	\noindent The generalized Jacobian of \(F\) is
	\begin{equation}
		\partial F(x)=A-D(Bx)B.
	\end{equation}
	
	\noindent Hence, the first step of Traub's method becomes
	\begin{equation}
		\begin{aligned}
			y_k
			&=x_k-\left(A-D(Bx_k)B\right)^{-1}
			\left(Ax_k-|Bx_k|-d\right).
		\end{aligned}
	\end{equation}
	
	\noindent Using the identity
	\[
	|Bx_k|=D(Bx_k)Bx_k,
	\]
	we obtain
	\[
	Ax_k-|Bx_k|
	=
	\left(A-D(Bx_k)B\right)x_k.
	\]
	Therefore,
	\begin{equation}
		\begin{aligned}
			y_k
			&=
			x_k-
			\left(A-D(Bx_k)B\right)^{-1}
			\left[
			\left(A-D(Bx_k)B\right)x_k-d
			\right] \\
			&=
			x_k-x_k+
			\left(A-D(Bx_k)B\right)^{-1}d.
		\end{aligned}
	\end{equation}
	
	Thus,
	\begin{equation}
		\boxed{
			y_k=
			\left(A-D(Bx_k)B\right)^{-1}d.
		}
	\end{equation}
	
	Similarly,
	$F(y_k)=Ay_k-|By_k|-d,$
	and
	$|By_k|=D(By_k)By_k.$
	
	Consequently,
	\begin{equation}
		\boxed{
			x_{k+1}
			=
			y_k-
			\left(A-D(Bx_k)B\right)^{-1}
			\left[
			\left(A-D(By_k)B\right)y_k-d
			\right].
		}
	\end{equation}
	
	\noindent Hence, the generalized Traub method for solving the NGAVE (\ref{Equ1}) is given by
	\begin{equation} \label{Traub's Method}
		\boxed{
			\begin{aligned}
				y_k
				&=
				\left(A-D(Bx_k)B\right)^{-1}d,\\
				x_{k+1}
				&=
				y_k-
				\left(A-D(Bx_k)B\right)^{-1}
				\left[
				\left(A-D(By_k)B\right)y_k-d
				\right].
			\end{aligned}
		}
	\end{equation}
	
	
	\begin{figure}[htbp]
		\centering
		\fbox{
			\begin{minipage}{0.95\textwidth}
				\textbf{Algorithm 1: Generalized Traub's Method (GTM)}
				
				\textbf{Input:} Matrices $A,B\in\mathbb{R}^{n\times n}$,
				vector $d\in\mathbb{R}^n$, initial guess $x_0$,
				tolerance $\varepsilon>0$, and maximum iterations $N_{\max}$.
				
				
				\begin{enumerate}
					\item Set $k=0$.
					
					\item Compute
					\[
					y_k=\left(A-D(Bx_k)B\right)^{-1}d.
					\]
					
					and
					\[
					\begin{aligned}
						x_{k+1}=y_k &-\left(A-D(Bx_k)B\right)^{-1}\left[Ay_k-|By_k|-d\right].
					\end{aligned}
					\]
					
					\item If
					\[
					\left\|Ax_{k+1}-|Bx_{k+1}|-d\right\|
					\leq \varepsilon,
					\]
					set $x^*=x_{k+1}$ and stop.
					
					\item Set $k:= k+1$. If $k<N_{\max}$,
					return to Step 2; otherwise, stop and return
					$x^*=x_{k+1}$.
				\end{enumerate}
			\end{minipage}
		}
	\end{figure}

	
	\noindent We now present the convergence analysis of the proposed generalized Traub's method (GTM).
	\begin{lemma} \label{Lemma 01}
		If minimum singular value of matrix $A$ is greater than maximum singular value of matrix $B$, then $(A-DB)$ is invertible for any digonal matrix $D$ whose diagonal whose diagonal entries comes from the set \{0, 1,-1\}.
	\end{lemma}
	\begin{proof}
		The proof follows from Lemma 2.1 of \cite{KhaldiNGAVE2024}.
	\end{proof}

	\begin{theorem} \label{Main Thm 01}
		Suppose that $||(A-DB)^{-1}|| < \frac{1}{3 ||B||}$ holds for each digonal matrix $D$ whose diagonal elements belong to $\{0,-1, 1\}$. Then the generalized Traub's method defined in (\ref{Traub's Method}) is converges linearly from any starting point to a solution $x^{*}$ of the NGAVE (\ref{Equ1}).
	\end{theorem}
	\begin{proof}
		By Lemma \ref{Lemma 01}, $(A-DB)^{-1}$ exists, therefore the generalized Traub's method is well defined.	
		
		\noindent Suppose that \(x^{*}\) is a solution of the NGAVE (\ref{Equ1}), then $Ax^{*}-|x^{*}|=d$ or $(A-D(Bx^{*})B)x^{*}=d$.
		
		\noindent Now, subtracting the above equation from the first step of GTM (\ref{Traub's Method}), i.e., $(A-D(Bx_k)B)y_k=d$. So we get \\
		$(A-D(Bx_k)B)y_k - (A-D(Bx^*)B)x^* =0,$ and \\
		$(A-D(Bx_k)B)y_k - (A-D(Bx_k)B)x^* + |Bx^*|-D(Bx_k)Bx^* =0,$ and \\
		$(A-D(Bx_k)B)(y_k-x^*)=D(Bx_k)Bx^*-|Bx^*|,$ and \\
		$(A-D(Bx_k)B)(y_k-x^*)=-D(Bx_k)(Bx_k-Bx^*) - |Bx^*|+|Bx_k|,$ and
		\begin{equation} \label{Proof 01}
			(y_k-x^*)=(A-D(Bx_k)B)^{-1}(|Bx_k| - |Bx^*|-D(Bx_k)(Bx_k-Bx^*)).
		\end{equation}
		Taking norm and using Theorem \ref{Result 2}, we will get
		\begin{equation} \label{Proof 02}
			||y_k-x^*|| \leq 3||(A-D(Bx_k)B)^{-1}|| (||B|| ||x_k-x^*||).
		\end{equation}	
		
		
		\noindent Now, from the second step of GTM (\ref{Traub's Method}), we have	
		\begin{equation} \label{Proof 03}	
			x_{k+1}=y_k-(A-D(Bx_k)B)^{-1}[(A-D(By_k)B)y_k-d].	
		\end{equation}		
		
		\noindent Since $x^*$ is the solution of NGAVE (\ref{Equ1}), then $Ax^* - |Bx^*|-d=0.$\\
		Further, $(A-D(By_k)B)y_k-d= Ay_k - |By_k|-d - 0,$ and\\
		$(A-D(By_k)B)y_k-d= Ay_k - |By_k|-d - Ax^* - |Bx^*|-d,$ and\\
		$(A-D(By_k)B)y_k-d= A(y_k - x^*) -(|By_k| - |Bx^*|).$\\
		Now, substituting this value into Equ. (\ref{Proof 03}), we get	
		\begin{equation} \label{Proof 04}	
			x_{k+1}=y_k-(A-D(Bx_k)B)^{-1}[A(y_k - x^*) -(|By_k| - |Bx^*|)].	
		\end{equation}	
		Now, subtracting $x^*$ in Equ. (\ref{Proof 04}), we get	 
		\begin{equation} \label{Proof 05}	
			x_{k+1}-x^*=y_k-x^*-(A-D(Bx_k)B)^{-1}[A(y_k - x^*) -(|By_k| - |Bx^*|)].	
		\end{equation}
		Now,
		\begin{equation*} \label{Proof 06}	
			x_{k+1}-x^*=(A-D(Bx_k)B)^{-1}(A-D(Bx_k)B)(y_k-x^*)-(A-D(Bx_k)B)^{-1}[A(y_k - x^*) -(|By_k| - |Bx^*|)],
		\end{equation*}
		and
		\begin{equation*} \label{Proof 07}	
			x_{k+1}-x^*=(A-D(Bx_k)B)^{-1} [(A-D(Bx_k)B)(y_k-x^*)-A(y_k - x^*) +|By_k| -|Bx^*|],	
		\end{equation*}
		and
		\begin{equation*} \label{Proof 08}	
			x_{k+1}-x^*=(A-D(Bx_k)B)^{-1} [-D(Bx_k)B(y_k-x^*) +|By_k| -|Bx^*|].	
		\end{equation*}
		Taking norms and using \(\|D(Bx_k)\|\leq1\), together with the inequality of Theorem \ref{Result 2}, we get
		\begin{equation} \label{Proof 09}	
			||x_{k+1}-x^*||  \leq 3||(A-D(Bx_k)B)^{-1}|| (||B||||y_k-x^*||).	
		\end{equation}
		Now combining Equ. (\ref{Proof 02}) and Equ. (\ref{Proof 09}), we get
		\begin{equation} \label{Proof 10}	
			||x_{k+1}-x^*||  \leq 9||(A-D(Bx_k)B)^{-1}||^{2}||B||^{2}(||x_k-x^*||).	
		\end{equation}	
		Since $||(A-DB)^{-1}|| < \frac{1}{3 ||B||}$, we have $9||(A-D(Bx_k)B)^{-1}||^{2}||B||^{2}<1$. Therefore, 
		\begin{equation} \label{Proof 11}	
			||x_{k+1}-x^*||  < ||x_k-x^*||.	
		\end{equation}	
		Hence, the sequence generated by the generalized Traub's method converges linearly to $x^*$ from any initial vector \(x_0\). 
	\end{proof}

	
	\noindent Now, we have the following sufficient condition for the linear convergence of the GTM (\ref{Traub's Method}).
	\begin{theorem} \label{Main Thm 02}
		Suppose that $||A^{-1}|| < \frac{1}{4 ||B||},$ $||B|| \ne 0$ and $D(Bx_k) \ne 0$. Then, for each $d$, the NGAVE (\ref{Equ1}) is uniquely solvable, and the generalized Traub's method defined in (\ref{Traub's Method}) converges linearly to the unique solution $x^*$ from any initial point.
	\end{theorem}
	\begin{proof}
		The existence of a unique solution of the NGAVE (\ref{Equ1}) follows from the Theorem \ref{Result 1} and the existence of $(A-DB)^{-1}$ are confirmed by Lemma 2.1 of \cite{KhaldiNGAVE2024}.	\\
		Now, we have $(A-DB)^{-1}=(I-A^{-1}DB)^{-1}A^{-1}$, and taking norm, we get\\
		\begin{equation} \label{Proof 001}
			||(A-DB)^{-1}||=||(I-A^{-1}DB)^{-1}||||A^{-1}||.
		\end{equation}
		Since $||A^{-1}DB||= ||A^{-1}||||B||<1,$ then by Theorem \ref{Result 3}, we have		
		\begin{equation} \label{Proof 002}	
			||I-A^{-1}DB||  \leq \frac{1}{1-||A^{-1}DB||} \leq \frac{1}{1-||A^{-1}||||B||}.	
		\end{equation}
		From Equ. (\ref{Proof 001}) and Equ. (\ref{Proof 002}), we have
		\begin{equation} \label{Proof 003}	
			||(A-DB)^{-1}|| \leq \frac{||A^{-1}||}{1-||A^{-1}||||B||}.	
		\end{equation}	
		So,	
		\begin{equation} \label{Proof 004}	
			||(A-DB)^{-1}|| \leq \frac{||A^{-1}||}{1-||A^{-1}||||B||} < \frac{\frac{1}{4||B||}}{1- \frac{1}{4}} < \frac{1}{3||B||}.
		\end{equation}	
	\end{proof}
	
	\noindent Therefore, by Theorem \ref{Main Thm 01}, the sequence generated by the GTM (\ref{Traub's Method}) converges linearly to the unique solution of the NGAVE (\ref{Equ1}) from any initial point.

	\begin{remark} \label{Remark 2}
		The conditions of Theorem \ref{Main Thm 01} and Theorem \ref{Main Thm 02} are also applicable to the generalized Newton method, as defined in \cite{KhaldiNGAVE2024} for finding the solution of NGAVE (\ref{Equ1}). The proof is similar to that of Theorem \ref{Main Thm 01} and Theorem \ref{Main Thm 02}, respectively.
	\end{remark}
	
	

	

\subsection{Modified Picard's Method} \label{section_Picard}
	
	In this section, we dicuss the modified Picard's method for computing the solution of the NGAVE. Khaldi et al. \cite{KhaldiNGAVE2024} proposed the following Picard iteration method for solving the NGAVE (\ref{Equ1}).
	$$x_{i+1}=A^{-1}|Bx_{i}|+A^{-1}d, ~~ i= 0,1,2, \cdots.$$

	
	\noindent Now, we will define the following Modified Picard's Method (MPM):
	\[
	\begin{cases}
		x_{i+\frac{1}{2}}=A^{-1}|Bx_{i}|+A^{-1}d,\\[2mm]
		x_{i+1}=A^{-1}|Bx_{i+\frac12}|+A^{-1}d.
	\end{cases}
	\qquad i=0,1,2,\cdots.
	\]
	
	
	\begin{center}
		\fbox{
			\begin{minipage}{0.99\textwidth}
				\textbf{Algorithm 2: Modified Picard's Method (MPM)}
				
				\textbf{Input:} $A,B\in\mathbb{R}^{n\times n}$, $d\in\mathbb{R}^n$,
				initial vector $x_0$, tolerance $\varepsilon>0$, and maximum
				iterations $N_{\max}$.
				
				
				\begin{enumerate}
					\item Set $k=0$.
					\item Compute
					\[
					x_{k+\frac12}=A^{-1}(|Bx_k|+d).
					\]
					and
					\[
					x_{k+1}=A^{-1}(|Bx_{k+\frac12}|+d).
					\]
					\item If
					\[
					\|Ax_{k+1}-|Bx_{k+1}|-d\|\leq\varepsilon,
					\]
					stop and return $x_{k+1}$.
					\item Set $k := k+1$. If $k<N_{\max}$,
					return to Step 2; otherwise stop.
				\end{enumerate}
			\end{minipage}
		}
	\end{center}
	
	\noindent We now establish the convergence properties of the proposed modified Picard's method (MPM).
	\begin{theorem} \label{Theorem 1}
		Suppose that the NGAVE (\ref{Equ1}) is solvable, where $A,B\in\mathbb{R}^{n\times n}$ and $A$ is nonsingular.
		
		If
		\[
		\rho\!\left((|A^{-1}||B|)^2\right)<1,
		\]
		then, for every initial vector
		$x^{(0)}\in\mathbb{R}^n$, the sequence
		$\{x^i\}$ generated by the modified Picard's iteration converges to the unique solution
		$x^*$ of the NGAVE (\ref{Equ1}).
	\end{theorem}
	\begin{proof}
		Suppose that \(x^{*}\) is a solution of the NGAVE (\ref{Equ1}).
		
		The modified Picard iteration is
		\[
		\begin{cases}
			x_{i+\frac12}
			=
			A^{-1}(|Bx_i|+d),\\[2mm]
			x_{i+1}
			=
			A^{-1}(|Bx_{i+\frac12}|+d).
		\end{cases}
		\]
		
		Now the fixed-point equation gives
		\begin{equation}
			x_{i+\frac12}-x^*
			=
			A^{-1}
			(
			|Bx_i|-|Bx^*|
			),
		\end{equation}
		and
		\begin{equation}
			x_{i+1}-x^*
			=
			A^{-1}
			(
			|Bx_{i+\frac12}|-|Bx^*|
			).
		\end{equation}
		
		Taking component-wise absolute values, we get
		\begin{equation} \label{001}
			|x_{i+\frac12}-x^*|
			=
			|A^{-1}
			(
			|Bx_i|-|Bx^*|
			)| 
			\le
			|A^{-1}|
			\,
			\bigl||Bx_i|-|Bx^*|\bigr|,
		\end{equation}
		and
		\begin{equation} \label{002}
			|x_{i+1}-x^*|
			=
			|A^{-1}
			(
			|Bx_{i+\frac12}|-|Bx^*|
			)|
			\le
			|A^{-1}|
			\,
			\bigl||Bx_{{i+\frac12}}|-|Bx^*|\bigr|.
		\end{equation}

		From Equ. (\ref{001}) and Equ. (\ref{002}), we get
		\[
		|x_{i+\frac12}-x^*|
		\le
		|A^{-1}|
		|B(x_i-x^*)|.
		\]
		
		Since
		\[
		|B(x_i-x^*)|
		\le
		|B||x_i-x^*|,
		\]
		it follows that
		\[
		|x_{i+\frac12}-x^*|
		\le
		|A^{-1}||B|
		|x_i-x^*|.
		\]
		
		Similarly,
		\[
		|x_{i+1}-x^*|
		\le
		|A^{-1}||B|
		|x_{i+\frac12}-x^*|.
		\]
		
		Combining the above inequalities, we get
		\[
		|x_{i+1}-x^*|
		\le
		(|A^{-1}||B|)^2
		|x_i-x^*|.
		\]
		
		Since,
		\[
		\rho\!\left((|A^{-1}||B|)^2\right)<1,
		\]
		then the sequence of the MPM converges to the unique solution $x^{*}$.

		
	\end{proof}
	
	\section{Numerical Results} \label{section_Numerical}
	This section presents numerical experiments to evaluate the effectiveness of the proposed methods and compare their computational performance with existing approaches. All numerical experiments were conducted in \textbf{MATLAB R2016a} on a computer equipped with an Intel(R) Core(TM) i5-10210U processor running at 1.60 GHz, 8 GB of RAM, and the Windows 11 operating system.
	The stopping criterion is based on the residual
	\[
	\mathrm{RES}=\|Ax^{(k)}-|Bx^{(k)}|-d\|,
	\]
	and the iteration is terminated whenever
	\[
	\mathrm{RES}\leq 10^{-6}
	\]
	or the number of iterations exceeds $1000$.
	
	
	The zero vector is used as the initial starting point $x_{0}$. In the numerical results tables, ``Iter", ``CPU" and ``Res" represent the number of iterations, the total computational time, and the residual error, respectively.
	We compare our methods with Picard's method \cite{KhaldiNGAVE2024}, generalized Newton's method \cite{KhaldiNGAVE2024}
	and SOR-like method \cite{ShiLiangWuNGAVE_SOR2023}.
	
	\begin{example} \label{Exam_1}
		Let us consider the following matrices $A,B\in \mathbb{R}^{7\times7}$ and vector $d\in\mathbb{R}^{7}$
		\begin{equation*}
			A=	
			\begin{bmatrix}
				5&1&0&0&0&0&0\\
				1&5&1&0&0&0&0\\
				0&1&5&1&0&0&0\\
				0&0&1&5&1&0&0\\
				0&0&0&1&5&1&0\\
				0&0&0&0&1&5&1\\
				0&0&0&0&0&1&5
			\end{bmatrix},
			~~~ B=
			\begin{bmatrix}
				\frac12 & \frac14 & 0 & 0 & 0 & 0 & 0\\
				\frac14 & \frac12 & \frac14 & 0 & 0 & 0 & 0\\
				0 & \frac14 & \frac12 & \frac14 & 0 & 0 & 0\\
				0 & 0 & \frac14 & \frac12 & \frac14 & 0 & 0\\
				0 & 0 & 0 & \frac14 & \frac12 & \frac14 & 0\\
				0 & 0 & 0 & 0 & \frac14 & \frac12 & \frac14\\
				0 & 0 & 0 & 0 & 0 & \frac14 & \frac12
			\end{bmatrix},
			~~~ and ~~~ d= 
			\begin{bmatrix}
				1\\
				1\\
				1\\
				1\\
				1\\
				1\\
				1
			\end{bmatrix}.
		\end{equation*}
	\end{example}
	Since $\sigma_{min}(A)=3.1522 > 0.9619 =\sigma_{max}(B).$ Then above NGAVE has a unique solution for all $d \in R^{7}.$
	The exact solution of the corresponding NGAVE is
	
	\[
	x^{*}
	=
	\begin{bmatrix}
		0.1953 &
		0.1618 &
		0.1675 &
		0.1664 &
		0.1675 &
		0.1618 &
		0.1953
	\end{bmatrix}^{T}.
	\]

	The numerical performances of Picard's iteration, the generalized Newton method, SOR-like method, modified Picard's method and generalized Traub's method are reported in Table~\ref{tab:comparison51}.
	As shown in Table~\ref{tab:comparison51}, the proposed methods remain computationally efficient.
	Although both GTM and the GNM converge in only two iterations, GTM requires the least CPU time (0.000396 s) while achieving a residual of the order of $10^{-16}$, indicating excellent numerical accuracy. Moreover, the modified Picard's Method reduces the iteration count from 8 to 4 compared with the classical Picard's method, confirming that the proposed modification significantly accelerates convergence. These results clearly show that the proposed methods are more efficient and competitive than the existing iterative schemes for solving the NGAVE.

	\begin{table}[htbp]
		\centering
		\caption{Numerical results for Example~\ref{Exam_1}.}
		\label{tab:comparison51}
		\begin{tabular}{lccc}
			\hline
			Methods & Iter & CPU Time (s) & Res \\
			\hline
			Picard's Method \cite{KhaldiNGAVE2024} & 8 & 0.001537 & $3.82497\times10^{-7}$ \\
			GNM \cite{KhaldiNGAVE2024} & 2 & 0.001081 & $4.00296\times10^{-16}$ \\
			SOR-like Method ($\rho=0.1405$, $\omega_0=1.0378$) \cite{ShiLiangWuNGAVE_SOR2023} & 6 &  0.003263 & $8.90681\times10^{-8}$ \\
			Proposed Modified Picard's Method & 4 & 0.001785 & $1.74717\times10^{-7}$ \\
			Proposed GTM & 2 & 0.000396 & $2.22045\times10^{-16}$ \\
			\hline
		\end{tabular}
	\end{table}
	
	
	
	\begin{example} \label{Exam_3} \cite{KhaldiNGAVE2024}
		The hydrodynamic equilibrium problem can be formulated as the following system of non-differentiable algebraic equations:	
		$$Cx+max(0,Bx)=c,$$
		where $B, C \in \mathbb{R}^{n \times n},$ $c \in \mathbb{R}^{n},$ and $max(0,Bx)$ is understood component-wise.\\
		Using $max(0,z)=\frac{1}{2}(z+|z|)$, we get the following type of NGAVE 
		$$Cx+\frac{1}{2}(Bx+|Bx|)=c,$$
		and $$(2C+B)x+|Bx|=2c,$$
		where $A=-(2C+B),$ $d=-2c.$
	\end{example}
	\noindent Now consider the following hydrodynamic equation where $B,$ $C$ and $c$ are given by
	\begin{equation*}
		B= tridiag(0.05,0.2,0.05) =
		\begin{bmatrix}
			0.2 & 0.05 & 0 & \cdots & 0 & 0\\
			0.05 & 0.2 & 0.05 & \cdots & 0 & 0\\
			0 & 0.05 & 0.2 & \cdots & 0 & 0\\
			\vdots & \vdots & \vdots & \ddots & \vdots & \vdots\\
			0 & 0 & 0 & \cdots & 0.2 & 0.05\\
			0 & 0 & 0 & \cdots & 0.05 & 0.2
		\end{bmatrix},
	\end{equation*}
	\begin{equation*}
		C = tridiag(-2.5,-25.5,-2.5) =
		\begin{bmatrix}
			-25.5 & -2.5 & 0 & \cdots & 0 & 0\\
			-2.5 & -25.5 & -2.5 & \cdots & 0 & 0\\
			0 & -2.5 & -25.5 & \cdots & 0 & 0\\
			\vdots & \vdots & \vdots & \ddots & \vdots & \vdots\\
			0 & 0 & 0 & \cdots & -25.5 & -2.5\\
			0 & 0 & 0 & \cdots & -2.5 & -25.5
		\end{bmatrix},
		~~~ and
		~~~ c=
		\begin{bmatrix}
			-1\\
			1\\
			-1\\
			1\\
			\vdots\\
			-1 \\
			1
		\end{bmatrix}.
	\end{equation*}
	\noindent So, $A=tridiag(4.95,50.8,4.95)$ and $d=[2,-2,\hdots,2,-2]^{T}.$\\
	Table~\ref{tab:Example423} shows the computational results for different problem dimensions.
	For all tested dimensions, the modified Picard's method and generalized Traub's method converge in just two iterations. While GTM regularly reaches residuals on the order of $10^{-15}$, indicating good numerical accuracy, MPM retains competitive CPU time as the issue dimension grows. As a result, GTM shows quick convergence with superior numerical precision, while MPM shows an effective balance between iteration count and computational cost.

	\begin{table}[htbp]
		\centering
		\caption{Numerical results for Example~\ref{Exam_3}.}
		\label{tab:Example423}
		\renewcommand{\arraystretch}{1.15}
		\begin{tabular}{ccccccc}
			\hline
			$n$ & Methods & Iter & CPU (s) & Res \\
			\hline
			\hline
			\multirow{5}{*}{100}
			& Picard's Method \cite{KhaldiNGAVE2024} &  3 & 0.001212 & $5.996289\times10^{-7}$ \\
			& GNM \cite{KhaldiNGAVE2024} & 2 & 0.000596 & $3.447059\times10^{-15}$ \\
			& SOR-like ($\rho=0.003937$, $\omega_0=1.000986$) \cite{ShiLiangWuNGAVE_SOR2023} & 4 & 0.000746 & $8.017673\times10^{-9}$ \\
			&Proposed Modified Picard's Method & 2 & 0.000498 &$2.950723\times10^{-9}$ \\
			&Proposed GTM & 2 & 0.000794 & $2.44249\times10^{-15}$ \\
			\hline
			
			\multirow{5}{*}{500}
			& Picard's Method \cite{KhaldiNGAVE2024} & 4 & 0.012378 & $6.543908\times10^{-9}$ \\
			& GNM \cite{KhaldiNGAVE2024} &  2 & 0.012127 & $7.822149\times10^{-15}$ \\
			& SOR-like ($\rho=0.003937$, $\omega_0=1.000986$) \cite{ShiLiangWuNGAVE_SOR2023} & 4 & 0.006998 & $4.98145\times10^{-7}$ \\
			&Proposed Modified Picard's Method & 2 & 0.005781 &$6.543908\times10^{-9}$ \\
			&Proposed GTM & 2 & 0.023428 & $5.068262\times10^{-15}$ \\
			\hline
			
			\multirow{5}{*}{1000}
			& Picard's Method \cite{KhaldiNGAVE2024} & 4 & 0.050001 & $9.244874\times10^{-9}$ \\
			& GNM \cite{KhaldiNGAVE2024} & 2 & 0.094929 & $1.109112\times10^{-14}$ \\
			& SOR-like ($\rho=0.003937$, $\omega_0=1.000986$) \cite{ShiLiangWuNGAVE_SOR2023} & 4 & 0.039213 & $2.514193\times10^{-8}$ \\
			&Proposed Modified Picard's Method & 2 & 0.037881 & $9.244874\times10^{-9}$ \\
			&Proposed GTM & 2 & 0.189858 & $7.1088996\times10^{-15}$ \\
			\hline
			
			\multirow{5}{*}{2000}
			& Picard's Method \cite{KhaldiNGAVE2024} & 4 & 0.300307 & $1.3067426\times10^{-8}$ \\
			& GNM \cite{KhaldiNGAVE2024} & 2 & 0.564395 & $1.569307\times10^{-14}$ \\
			& SOR-like ($\rho=0.003937$, $\omega_0=1.000986$) \cite{ShiLiangWuNGAVE_SOR2023} & 4 & 0.287949 & $3.553931\times10^{-8}$ \\
			&Proposed Modified Picard's Method & 2 & 0.264036 & $1.306743\times10^{-8}$ \\
			&Proposed GTM & 2 & 1.069536 & $9.992007\times10^{-15}$ \\
			\hline
		\end{tabular}
	\end{table}
	

	\begin{example} \label{Exam_2} \cite{ShiLiangWuNGAVE_SOR2023}
		Consider the following matrices $A,B\in\mathbb{R}^{n\times n}$ and $d\in\mathbb{R}^{n}$, as follows
		\begin{equation*}
			A=
			\begin{bmatrix}
				4 & -1 & 0 & \cdots & 0 & 0\\
				-1 & 4 & -1 & \cdots & 0 & 0\\
				0 & -1 & 4 & \cdots & 0 & 0\\
				\vdots & \vdots & \vdots & \ddots & \vdots & \vdots\\
				0 & 0 & 0 & \cdots & 4 & -1\\
				0 & 0 & 0 & \cdots & -1 & 4
			\end{bmatrix},
			~~~B=
			\begin{bmatrix}
				\frac12 & \frac14 & 0 & \cdots & 0 & 0\\
				\frac14 & \frac12 & \frac14 & \cdots & 0 & 0\\
				0 & \frac14 & \frac12 & \cdots & 0 & 0\\
				\vdots & \vdots & \vdots & \ddots & \vdots & \vdots\\
				0 & 0 & 0 & \cdots & \frac12 & \frac14\\
				0 & 0 & 0 & \cdots & \frac14 & \frac12
			\end{bmatrix},
			~~~ and
			~~~ d=
			\begin{bmatrix}
				1\\
				1\\
				1\\
				\vdots\\
				1 \\
				1
			\end{bmatrix}.
		\end{equation*}
	\end{example}
	
	\noindent The numerical results presented in Table~\ref{tab:Example42} shows that the proposed generalized Traub's method (GTM) consistently converges in only two iterations for all tested problem sizes, while maintaining residuals of the order of $10^{-15}$. The modified Picard's method also exhibits stable convergence, requiring only 10 iterations regardless of the matrix dimension, which is a significant improvement over the classical Picard's method and the SOR-like method. 
	The SOR-like method requires approximately 24–25 iterations for all problem sizes, whereas the modified Picard's method converges in only 10 iterations, reducing the iteration count by more than 50 \%. Moreover, the modified Picard's method consistently achieves lower CPU times than the SOR-like method as the matrix dimension increases, while both methods satisfy the stopping criterion.
	These results confirm the robustness and effectiveness of the proposed iterative schemes for solving large-scale NGAVE problems. 
	
	
	\begin{table}[htbp]
		\centering
		\caption{Numerical results for Example~\ref{Exam_2}.}
		\label{tab:Example42}
		\renewcommand{\arraystretch}{1.15}
		\begin{tabular}{ccccccc}
			\hline
			$n$ & Methods & Iter & CPU (s) & Res \\
			\hline
			
			\hline
			\multirow{5}{*}{1000}
			& Picard's Method \cite{KhaldiNGAVE2024} & 25 & 0.324832 & $9.36366\times10^{-7}$ \\
			& GNM \cite{KhaldiNGAVE2024} & 2 & 0.110854 & $8.6062\times10^{-15}$ \\
			& SOR-like ($\rho=0.0999$, $\omega_0=1.0263$) \cite{ShiLiangWuNGAVE_SOR2023} & 24 & 0.262963 & $4.98145\times10^{-7}$ \\
			&Proposed Modified Picard's Method & 13 & 0.298661 &$4.68118\times10^{-7}$ \\
			&Proposed GTM & 2 & 0.201835 & $2.26714\times10^{-15}$ \\
			\hline
			
			\multirow{5}{*}{2000}
			& Picard's Method \cite{KhaldiNGAVE2024} & 26 & 1.745393 & $6.64213\times10^{-7}$ \\
			& GNM \cite{KhaldiNGAVE2024} & 2 & 0.579081 & $1.16489\times10^{-14}$ \\
			& SOR-like ($\rho=0.0999$, $\omega_0=1.0263$) \cite{ShiLiangWuNGAVE_SOR2023} & 24 & 1.649509 & $7.068567\times10^{-7}$ \\
			&Proposed Modified Picard's Method & 13 & 1.733036 & $6.64213\times10^{-7}$ \\
			&Proposed GTM & 2 & 1.118929 & $2.26714\times10^{-15}$ \\
			\hline
			
			\multirow{5}{*}{3000}
			& Picard's Method \cite{KhaldiNGAVE2024} & 26 & 4.91361 & $8.14386\times10^{-7}$ \\
			& GNM \cite{KhaldiNGAVE2024} & 2 & 1.942717 & $1.40473\times10^{-14}$ \\
			& SOR-like ($\rho=0.0999$, $\omega_0=1.0263$) \cite{ShiLiangWuNGAVE_SOR2023} & 24 & 5.307157 & $8.66686\times10^{-7}$ \\
			&Proposed Modified Picard's Method & 13 & 4.857706 & $8.14386\times10^{-7}$ \\
			&Proposed GTM & 2 & 3.189991 & $2.26714\times10^{-15}$ \\
			\hline
			
			\multirow{5}{*}{4000}
			& Picard's Method \cite{KhaldiNGAVE2024} & 26 & 10.39578 & $9.40887\times10^{-7}$ \\
			& GNM \cite{KhaldiNGAVE2024} & 2 & 3.930474 & $1.60921\times10^{-14}$ \\
			& SOR-like ($\rho=0.0999$, $\omega_0=1.0263$) \cite{ShiLiangWuNGAVE_SOR2023} & 25 & 11.873479 & $4.73182\times10^{-7}$ \\
			&Proposed Modified Picard's Method & 13 & 10.972403 & $9.40887\times10^{-7}$ \\
			&Proposed GTM & 2 & 6.960577 & $2.26714\times10^{-15}$ \\
			\hline
			
			\multirow{5}{*}{5000}
			& Picard's Method \cite{KhaldiNGAVE2024} & 27 & 20.112418 & $5.26131\times10^{-7}$ \\
			& GNM \cite{KhaldiNGAVE2024} & 2 & 7.451955 & $1.79049\times10^{-14}$ \\
			& SOR-like ($\rho=0.0999$, $\omega_0=1.0263$) \cite{ShiLiangWuNGAVE_SOR2023} & 25 & 21.36423 & $5.29214\times10^{-7}$ \\
			&Proposed Modified Picard's Method & 14 & 20.349476 & $2.63058\times10^{-7}$ \\
			&Proposed GTM & 2 & 12.629602 & $2.26714\times10^{-15}$ \\
			\hline
		\end{tabular}
	\end{table}
	
	

	
	

	\section{Conclusions} \label{section_Conclusions}
	In this paper, we proposed two iterative methods, namely the modified Picard's method (MPM) and the generalized Traub's method (GTM), for solving the NGAVE. Under suitable assumptions, we established the well-definedness and convergence of the proposed methods. Numerical comparisons demonstrate that the proposed methods are computationally efficient and robust.
	The modified Picard’s method consistently reduces the number of iterations compared with the classical Picard method. 
	Whereas, the generalized Traub’s method consistently converges in two iterations across all three examples and produces residuals close to machine precision, demonstrating its rapid convergence and good numerical accuracy.
	

\paragraph{Acknowledgments.} 
The author, Shubham Kumar, gratefully acknowledges the financial support provided by the ABV–Indian Institute of Information Technology and Management, Gwalior, under the Institute Post-Doctoral Fellowship (IPDF), Application No. IPDF-Jan-2026-ES-02.



\bibliographystyle{abbrv}
\bibliography{NGAVE_ref}

\end{document}